\documentclass{amsart} 

\usepackage[normalem]{ulem}

\theoremstyle{plain}
\newtheorem{theorem}{Theorem}[section]

\newtheorem{proposition}[theorem]{Proposition}

\theoremstyle{definition}
\newtheorem{definition}[theorem]{Definition}
\newtheorem{example}[theorem]{Example}

\theoremstyle{remark}

\newenvironment{ack}{\noindent{\bf Acknowledgements}.}{}
\def\ch{\mbox{char}}
\def\EE{{\mathbb E}}
\def\VV{V(\underline{\bf X}, \Gamma)}
\def\NN{{\mathbb N}}
\def\ZZ{{\mathbb Z}}
\def\L{{\mathcal L}}

\title{Tropicalization and irreducibility of Generalized Vandermonde
Determinants}

\author{Carlos D'Andrea}
\address{Universitat de Barcelona, Departament d'{\`A}lgebra i Geometria.
Gran Via 585, E-08007 Barcelona, Spain.}
\email{cdandrea@ub.edu}
\urladdr{http://carlos.dandrea.name/}
\thanks{C. D'Andrea is supported
by the Programa Ram{\'o}n y Cajal, and the Research Project MTM2007--67493, both
from the
Ministerio de Educaci{\'o}n y Ciencia, Spain.}

\author{Luis Felipe Tabera}
\address{Universidad de Cantabria, Departamento de Matem\'aticas, Estad\'\i{}stica y Computaci\'on. Av los Castros s/n, E-39005, Santander, Spain}
\urladdr{http://personales.unican.es/taberalf/}
\thanks{L.F. Tabera is supported by the Research Project MTM2005-08690-CO2-02}


\begin{document}
\begin{abstract}
We find geometric and arithmetic conditions in order to characterize
the irreducibility of the determinant of the generic Vandermonde
matrix over the algebraic closure of any field $k$. We also
characterize those determinants whose tropicalization with respect
to the variables of a row is irreducible.
\end{abstract}

\maketitle

\section{Introduction}
Let $n, \, N$ positive integers, ${\bf X}_1, \ldots,
{\bf X}_N$ $n$-tuples of indeterminates, i.e.
\[{\bf X}_i:=\big(X_{i1}, \ldots, X_{in}\big)\quad i=1, \ldots, N\]
where each $X_{ij}$ is an indeterminate, and
$\Gamma:=\big(\gamma_1,\ldots,\gamma_N\big)$, a $N$-tuple of vectors in $\NN^n$,
$\gamma_j =(\gamma_{j1}, \ldots, \gamma_{jn})$. Set

\[\VV:=
\begin{vmatrix}
X_{11}^{\gamma_{11}} X_{12}^{\gamma_{12}}\cdots X_{1n}^{\gamma_{1n}} &
\ldots\ldots & X_{11}^{\gamma_{N1}} X_{12}^{\gamma_{N2}}\cdots
X_{1n}^{\gamma_{Nn}}\\\mbox{}\\
X_{21}^{\gamma_{11}} X_{22}^{\gamma_{12}}\cdots X_{2n}^{\gamma_{1n}} &
\ldots\ldots & X_{21}^{\gamma_{N1}} X_{22}^{\gamma_{N2}}\cdots
X_{2n}^{\gamma_{Nn}}\\\mbox{}\\
\multicolumn{3}{c}{\ldots\ldots\ldots}\\\mbox{}\\
X_{N1}^{\gamma_{11}} X_{N2}^{\gamma_{12}}\cdots X_{Nn}^{\gamma_{1n}} &
\ldots\ldots & X_{N1}^{\gamma_{N1}} X_{N2}^{\gamma_{N2}}\cdots
X_{Nn}^{\gamma_{Nn}}
\end{vmatrix}
\]
We call the polynomial $\VV\in \ZZ\big[(X_{i, j})_{1 \leq i\leq N, 1\leq j
\leq n}\big]$ the \textit{Generalized Vandermonde determinant} associated to
$\Gamma$.

\begin{example}
If $n=1$ and $\Gamma=(0, 1, \ldots, N-1)$, then \[\VV=\pm\prod_{1\leq i<j\leq
N}(X_{i1}-X_{j1}), \] the classical Vandermonde determinant.
\end{example}

\begin{example}\label{unigen}
It is a classical result (see for instance \cite{mac}) that if
$n=1$, then for any set $\Gamma\subset \NN$ of $N$ elements, the
determinant $\VV$ is a multiple of the classical Vandermonde
determinant $\prod_{1\leq i<j\leq N}(X_{i1}-X_{j1})$.
\end{example}

\begin{example}\label{char}
Suppose $n=2$, $N=3$ and $\Gamma:=((2, 0), (0, 2), (2, 2))$.

By computing the $3\times 3$ determinant we have that
$\VV$ is equal to
\[ X_{11}^2 X_{22}^2 X_{31}^2 X_{32}^2  - X_{11}^2  X_{32}^2  X_{21}^2  X_{22}^2
 - X_{12}^2  X_{21}^2  X_{31}^2  X_{32}^2\]
\[+ X_{12}^2  X_{31}^2  X_{21}^2  X_{22}^2  + X_{11}^2  X_{12}^2  X_{21}^2
X_{32}^2 - X_{11}^2  X_{12}^2  X_{31}^2  X_{22}^2.\]
Set $\Gamma':=\{(1, 0), (0, 1), (1, 1)\}$, it is easy to see that
\begin{itemize}
\item If $\ch(k)\neq2, $ then $\VV$ is absolutely irreducible over $k\big[(X_{i,
j})_{1 \leq i\leq N, 1\leq j
\leq n}\big]$(i.e. irreducible
over $\overline{k}\big[(X_{i, j})_{1 \leq i\leq N, 1\leq j
\leq n}\big]$, $\overline{k}$ being the algebraic closure of $k$).
\item If $\ch(k)=2$, then $\VV={V(\underline{\bf X}, \Gamma')}^2$ in
$k\big[(X_{i, j})_{1 \leq i\leq N, 1\leq j
\leq n}\big]$.
\end{itemize}

\end{example}

\begin{example}\label{firststep}
As an easy exercise, it can be proved that if $\Gamma\subset\NN^n$
is contained in an affine line, then $\VV$ factorizes in a similar
way as in the Example \ref{unigen}.
\end{example}

In the univariate case ($n=1$), the Vandermonde determinant is associated with
the interpolation problem, and it has
been extensively studied (see \cite{GV,gow,elm} and the references therein). The
multivariate interpolation
problem is naturally associated with
Generalized Vandermonde determinants, and there is also an extensive and current
literature
on the topic. See for instance
\cite{CL, GS, LS, olv, zhu}.

The purpose of this article is to study the irreducibility of $\VV$. As Example
\ref{char} suggests, the answer
will depend on the characteristic of $k$. Also, our intuition with the
univariate case may lead us to believe that
Generalized Vandermonde determinants have lots of irreducible factors.
Our main result essentially tell us that in general these polynomials are
absolutely irreducible.

There are some trivial factors that can be already read from the set of
exponents.
Let $\overline{\gamma}:= (g_1\ldots, g_n)$ where each $g_i$ is defined as
$\min\{{\gamma_1}_i, \ldots, {\gamma_N}_i \}$, $i=1,\ldots, n.$
It is easy to check that the following factorization holds:
\[V\big( \underline{\bf X}, (\gamma_1 , \ldots, \gamma_N ) \big)=
\prod_{i=1}^n \bigg( \prod_{j=1}^N X_{ij} \bigg)^{g_i} V \big( \underline{\bf
X}, (\gamma_1 -\overline{\gamma},
\ldots, \gamma_N- \overline{\gamma}) \big),\]
and $V \big( \underline{\bf X}, (\gamma_1 -\overline{\gamma}, \ldots, \gamma_N-
\overline{\gamma}) \big)$
has no monomial factor. Let $d_\Gamma$ be the largest integer such that
$\frac{1}{d_\Gamma} \{\gamma_1 -\overline{\gamma}, \ldots, \gamma_N
-\overline{\gamma}\} \subset \NN^n$, and  $\L_\Gamma\subset \mathbb{R}^n$ the
affine subspace spanned by $\Gamma$.

\begin{theorem}\label{main}
Let $N\geq 3$. The Vandermonde polynomial $\VV$ is irreducible in
$\overline{k}\big[(X_{ij})_{1
\leq i, j, \leq N}\big]$ if and only if the following three conditions apply:
\begin{itemize}
\item $\dim(\L_\Gamma)\geq2$,
\item $\gcd\big({\bf X}^{\gamma_i}\big)_{1\leq i\leq N}=1$, equivalently
$\overline{\gamma}=(0,\ldots,0)$,
\item $\ch(k)$ does not divide $d_\Gamma$.
\end{itemize}
\end{theorem}

Note that $\dim(\L_\Gamma)\geq2$ implies $N\geq3$ and $n>1$.
When $n=2$ and $N=3, 4$, Theorem \ref{main} can also be obtained from an
application of Ostrowski's work
\cite{ost} on the irreducibility of fewnomials  (see also \cite{BP}). We will
prove the general case by making use of
Bertini's Theorem on the variety defined by $\VV$, and applying some results
concerning algebraic independence of maximal
Vandermonde minors gotten in \cite{tab}.

When dealing with the problem of  factorizing multivariate polynomials, several
approaches like those given in \cite{ost, Gao} focus on the irreducibility of
the Newton polytope of the polynomial with respect to the operation of computing
Minkowski sums, which gives sufficient conditions to show irreducibility. A
refinement of this method can be obtained with the aid of Tropical Geometry:
instead of working with Newton polytopes, we can study regular subdivisions of
them. So, we can cover more general families of polynomials, but at the cost of
losing  track of the characteristic of the ground field $k$.

In our case, the problem can be dealt as follows: by expanding the
determinant of the generalized Vandermonde determinant with respect
to the first row, we get the following expansion:
$\VV=\sum_{i=1}^N(-1)^{\sigma_i}\Delta_i{\bf X}_1^{\gamma_i},$ For
the irreducibility problem, it is enough to consider a dehomogenized
version of $\VV$ as follows $$ \VV_{\bf aff}:={\bf X}_1^{\gamma_N}
+\sum_{i=1}^{N-1} A_i {\bf X}_1^{\gamma_i},$$ where
$A_i:=(-1)^{\sigma_i}\frac{\Delta_i}{\Delta_N},\ i=1,\ldots N-1$.

We can then regard $\VV_{\bf aff}$ as a polynomial in $K[{\bf X}_1]$,  $K$ being
now a field containing all the
$A_i$'s, $1,\ i=1,\ldots, N-1$.

Given any rank one valuation: $v:K\rightarrow \mathbb{R}$, we can extend it to
$K^n$ componentwise
as follows \[\begin{matrix}{\bf v}:& K^n & \rightarrow & \mathbb{R}^n \\&(z_1,
\ldots, z_n) &\mapsto & (v(z_1), \ldots, v(z_n))\end{matrix}.\]
The tropicalization of $\VV$ is then defined as
\[\textrm{Trop} (\VV)=\overline{{\bf v}(\{\VV_{\bf aff}=0\})} \subseteq
\mathbb{R}^n,\]
where the closure in the right hand side is taken with respect to the standard
topology in $\mathbb{R}^n$.

It turns out (see for instance, \cite{Bieri-Groves} or
\cite{einsiedler-2004-}) that
$\textrm{Trop}(\VV)$ is a connected polyhedral complex of
codimension $1$. If $\VV_{\bf aff}$ is reducible over $K[{\bf X}_1]$,
then Trop$(\VV)$ is a reducible tropical hypersurface, i.e. it can be
expressed as the union of two proper tropical hypersurfaces. So, if we
prove that for a special valuation ${\bf v}$, Trop$(\VV)$ is irreducible then
$\VV$ will be irreducible over any field $k$.

\begin{theorem}\label{irred_tropical}
Let $N\geq 3$. Given any field $K\supseteq k(A_1,\ldots,
A_{N-1})$, there exists a valuation $v$ defined over $K$ such that
$\textrm{Trop}(\VV)$ is an irreducible tropical variety if and only if:
\begin{itemize}
\item $\dim(\L_\Gamma)\geq2$,
\item $\gcd\big({\bf X}^{\gamma_i}\big)_{1\leq i\leq N}=1$, equivalently
$\overline{\gamma}=(0,\ldots,0)$.
\item $d_\Gamma=1$.
\end{itemize}
\end{theorem}

This result is optimal in the following sense: it is known
that Trop$(\{f=0\})$ does not depend on the field, but only on the
values $v(A_i)$'s.
\par
Take for instance $f=\sum_{i=0}^N A_i {\bf Y}^{i}$
over a field of characteristic zero and $d_\Gamma>1$, and let $g:=\sum_{i=0}^N
B_i {\bf Y}^{i}$ be a
polynomial with the same support with coefficients in a field of
characteristic $p|d_\Gamma$. Give to these polynomials valuations $v$ and $v'$
such that
$v'(B_i)=v(A_i)$. In these conditions, we will have
$$\mbox{Trop}(\{g=0\})=\mbox{Trop}(\{f=0\}),$$ but $g=(\sum_{i=0}^N B_i^{1/p}
{\bf Y}^{i/p})^p$ factorizes in the algebraic closure of its field
of definition, hence Trop$(\{f=0\})$ will always be reducible. So,
the tropical criteria will not  help to deduce the irreducibility of
$f$.

The paper is organized as follows: in Section \ref{bertini}, we give explicit
conditions on the irreducibility of the Vandermonde variety. In Section
\ref{sec:proof} we prove Theorem \ref{main}. We conclude by introducing some
tools from Tropical Geometry and by proving Theorem \ref{irred_tropical} in
Section \ref{sec:trop}.

\begin{ack}
We thank David Cox for posing us the problem of the irreducibility
of Vandermonde determinants, J.~I. Burgos and J.~C. Naranjo for useful
discussions, and the anonymous referee for helpful suggestions in order to
improve the
presentation of this work.
\end{ack}

\section{Bertini's Theorem and the irreducibility of the Vandermonde
variety}\label{bertini}
We begin by studying the geometric irreducibility of the variety
defined by $\VV$ in $\overline{k}^{Nn}$. In order to do this, we
will apply one of the several versions of Bertini's theorem given in
\cite{jou}. Recall that (\cite[Definition 4.1]{jou})
\begin{definition}
A $k$-scheme $\mathcal{V}$ over a field $k$ is said to be \textit{geometrically
irreducible} if $\mathcal{V}\otimes_k\overline{k}$ is an irreducible scheme.
\end{definition}

Now we are ready to present the version of Bertini's theorem that we will use.

\begin{theorem}\cite[Th\'eor\`eme 6.3]{jou}\label{ber}
Let $k$ be an infinite field, $\mathcal{V}$ a geometrically irreducible
$k$-scheme of finite type,
$\EE^m_k$ be the affine space
of dimension $m$, and $f:\mathcal{V}\to \EE^m_k$ a $k$-morphism, i.e.
\[
\begin{array}{ccccc}
f&:&\mathcal{V}&\to&\EE^m_k\\
&&z&\mapsto&\big(f_1(z), \ldots, f_m(z)\big)
\end{array}
\]
with $f_i\in\Gamma(\mathcal{V}, {\mathcal O}_{\mathcal{V}})$. If $\mathcal{V}$
is geometrically irreducible and $\dim\big(\overline{f(\mathcal{V})}\big)\geq
2$, then for almost all $\xi\in k^{m+1}$, $f^{-1}\left(\{z\in \EE^m_k:\,
\xi_0+\xi_1z_1+\ldots+\xi_nz_n=0\}\right)$ is geometrically irreducible.
\end{theorem}

\begin{definition}
Let $\Gamma=\{\gamma_1, \ldots, \gamma_N\}\subset\NN^n$, ${\bf Y}:=(Y_1, \ldots,
Y_n)$ a set of $n$ variables and $A_i, \, 1\leq i\leq N$ another set of
indeterminates. The \textit{generic polynomial supported in} $\Gamma$ is defined
as
\[P({\bf Y}, \Gamma):=\sum_{i=1}^NA_i\, {\bf Y}^{\gamma_i}.\]
\end{definition}

\begin{proposition}\label{jouap}
If $\dim(\L_\Gamma)\geq2$  and  $\gcd({\bf X}^{\gamma_i})_{1\leq i\leq N}=1$,
then $P({\bf Y}, \Gamma)$ defines an irreducible set in
$\overline{k(A_1,\ldots,A_N)}^n$.
\end{proposition}

\begin{proof}
Note that $\dim(\L_\Gamma)\geq2$ implies $N\geq3$, and moreover that there are
three components of $\Gamma$ that are not collinear. We pick a triple of vectors
with this property which we suppose w.l.o.g. that they are $\gamma_1, \,
\gamma_2, \, \gamma_3$.

In order to use Theorem \ref{ber}, let $\mathcal{V}:=\mbox{Spec} \left (k \big
[Y_1^{\pm1}, \ldots, Y_n^{\pm1} \big] \right)$ be the torus $(k^*)^n$, and set
\[\begin{array}{ccccc}
f&:&\mathcal{V}&\to&\EE^{N-1}_k\\ \\
&&z&\mapsto&\left(z^{\gamma_2-\gamma_1}, z^{\gamma_3-\gamma_1}, \dots,
z^{\gamma_N-\gamma_1}\right).
\end{array}\]

By hypothesis, the rank of the matrix $\begin{pmatrix} \gamma_2-\gamma_1\\
\gamma_3-\gamma_1\\ \ldots\\
\gamma_N-\gamma_1\end{pmatrix} $ is at least two. So, the top two by two
submatrix of its Smith normal
form is
$\begin{vmatrix}d_1 & 0\\ 0 & d_2\end{vmatrix}=d_1d_2\neq 0$. Hence, it follows
that, under a suitable monomial
change of coordinates, the map $f$ is of the form $z\mapsto (z_1^{d_1},
z_2^{d_2}, \ldots)$
and this shows that the dimension of the image of $f$ is greater than one, so we
can apply Bertini's Theorem and have that,
for almost all $\xi \in \overline{k}^{N}$, the polynomial $Q(\xi, {\bf Y}):=
\sum_{i=1}^N \xi_i {\bf Y}^{\gamma_i}$ defines
an irreducible set in $(\overline{k}^*)^n$. The fact that the $\gcd\big({\bf
Y}^{\gamma_i}\big)_{1\leq i\leq N}=1$ implies that $Q(\xi, {\bf Y})$ defines an
irreducible set also in $\overline{k}^n$ for almost all $\xi$, and hence the
claim holds for $P({\bf Y}, \Gamma)$.

\end{proof}

\begin{proposition}\label{desc}
Let $\Gamma=\big(\gamma_1, \ldots, \gamma_N\big)\subset{\NN^n}$ with $N \geq 3$
and $\gcd\big({\bf Y}^{\gamma_i}\big)_{ 1\leq i\leq N}=1$.
Then
\begin{itemize}
\item If $\dim(\L_\Gamma)=1$, then $P({\bf Y}, \Gamma)$ factorizes
in $\overline{k(A_1, \ldots, A_N)}[{\bf Y}]$.
\item If $\dim(\L_\Gamma)>1$, then
\begin{itemize}
\item if $\ch (k)$ does not divide $d_\Gamma$, then $P({\bf Y}, \Gamma)$ is
absolutely irreducible.
\item If $\ch (k)=p\, |\, d_\Gamma$, then $P({\bf Y}, \Gamma)= R({\bf
Y})^{p^r}$, with $p^r\, |\, d_\Gamma$, $p^{r+1}$
not dividing $d_\Gamma$,
and $R({\bf Y})\in\overline{k(A_1, \ldots, A_N)}[{\bf Y}]$  irreducible of
support $\frac{1}{p^r}\Gamma$.
\end{itemize}
\end{itemize}
\end{proposition}

\begin{proof}
If the vertices are contained in an affine line, then, via a monomial
transformation, we can reduce $P({\bf Y}, \Gamma)$ to a univariate polynomial,
which always factorizes (due to the fact that $N>2$) as a product of linear
factors with coefficients in $\overline{k(A_1, \ldots, A_N)}.$ The variety
defined by this polynomial may be reducible or not, depending on the
inseparability of this polynomial.

Suppose now that $\L_\Gamma$ has affine dimension at least two. Then we can
apply
the previous proposition, and conclude that the variety defined
by $P({\bf Y}, \Gamma)$ is irreducible over $\overline{k(A_1, \ldots, A_N)}$.

Hence, there exists an irreducible polynomial $R({\bf Y})\in\overline{k(A_1,
\ldots, A_N)}[{\bf Y}]$
and $D\in\NN$ such that
$$
P({\bf Y}, \Gamma)= R({\bf Y})^{D}.
$$

It is clear that $R({\bf Y})$ cannot be a monomial. If $D=1$ then we are done.
Suppose then $D>1$ and $p:=\ch (k)$  coprime with $d_\Gamma$. We can suppose
w.l.o.g. that $p$ does not divide the
first coordinate of one of the $\gamma_{i}$'s. But then, we have
\begin{equation}\label{tit}
0\neq\frac{\partial P({\bf Y}, \Gamma)}{\partial {\bf
Y}_1}=\sum_{i=1}^N\gamma_{i1}A_i\, {\bf Y}^{\gamma_i-{\bf e}_1}=DR({\bf
Y})^{D-1}\frac{\partial R({\bf Y})}{\partial {\bf Y}_1}.
\end{equation}
In particular, we get that $p$ does not divide $D$. As $R({\bf Y})$ is not a
monomial, it turns out
that $\frac{\partial P({\bf Y}, \Gamma)}{\partial {\bf Y}_1}$ has at least two
nonzero terms.
This implies that $P({\bf Y}, \Gamma)$ has at least two different monomials with
positive degree on ${\bf Y}_1$, so the degree of $R({\bf Y})$ with respect to
the first variable is positive and, due to (\ref{tit}), the same applies to
$\frac{\partial P({\bf Y}, \Gamma)}{\partial {\bf Y}_1}$.
\par
We can then eliminate ${\bf Y}_1$ by computing the univariate (or classical)
resultant (see \cite{GKZ}) of
the polynomials $P({\bf Y}, \Gamma)$ and $\frac{\partial P({\bf Y},
\Gamma)}{\partial {\bf Y}_1}$ with
respect to the first variable. This is a polynomial in $k[A_1,\ldots,A_n,{\bf
Y}_2,\ldots,{\bf Y}_n]$ which must be identically zero, as $R({\bf Y})$ is a
common factor of both $P({\bf Y}, \Gamma)$ and $\frac{\partial P({\bf Y},
\Gamma)}{\partial {\bf Y}_1}$. This is a contradiction with the fact that
$A_1,\ldots,A_n,{\bf Y}_2,\ldots,{\bf Y}_n$ are algebraically independent.

\par
Suppose now that $p:=\ch (k)\, |\, d_\Gamma$, let $r$ be the maximum such that
$p^r\,
|\, D$, so we can write
$D=p^r\,q$ with $(p,q)=1$. Write $R({\bf Y})=\sum_{j=1}^M
R_j {\bf
Y}^{\gamma'_j}$. We have then
\[P({\bf Y}, \Gamma)=\sum_{i=1}^NA_i\, {\bf Y}^{\gamma_i}=\left(\sum_{j=1}^M R_j
{\bf Y}^{\gamma'_j} \right)^{p^rq} =\left(\sum_{j=1}^M {R_j}^{p^r} {\bf
Y}^{p^r\gamma'_j} \right)^{q}.\]
From here, we deduce that $p^r$ divides $d_\Gamma$. Moreover, after dividing all
the exponents by $p^r$ we get
\[P_r({\bf Y}, \Gamma):=\sum_{i=1}^NA_i\, {\bf Y}^{\frac{\gamma_i}{p^r}}=
\left(\sum_{j=1}^M {R_j}^{p^r} {\bf Y}^{\gamma'_j} \right)^{q}.\]
It turns out that $P_r({\bf Y}, \Gamma)=P({\bf Y},\frac{1}{p^r}\Gamma)$. An
argument like above over $\frac{1}{p^r}\Gamma$
shows that $q$ cannot be different  than one if the $A_i$ are algebraically
independent. This completes the proof.
\end{proof}

As in the introduction, for $\ell=1, \ldots, N$ we set
\[\Delta_\ell( \underline{{\bf X}}, \Gamma):= \det \left({{\bf X}_i}^{\gamma_j}
\right)_{\begin{array}{l} 2 \leq i\leq N, 1\leq j\leq N\\
j\neq\ell\end{array}},\]
i.e. $\Delta_\ell$ is
the minor obtained by deleting the first row and $\ell$-th column in the
Generalized Vandermonde matrix.

\begin{theorem}\label{minors}
For any index $\ell_0$, the family $\{\Delta_\ell/\Delta_{\ell_0}\ :\ \ell=1,
\ldots N, \ell\neq \ell_0\}$ is algebraically independent over any field $k$.
\end{theorem}

\begin{proof}
Suppose without loss of generality that $\ell_0=N$. It is easy to see that
$\Delta_N$ does not
define the zero function on $\overline{k}^{(n-1)N}$ even if $\ch (k)>0$.

Let $\mathcal{V}$ be the Zariski image of the rational map:
\[\begin{matrix}
\overline{k}^{n(N-1)}&\longrightarrow & \overline{k}^{n(N-1)+(N-1)}&\\
({\bf X}_2, \ldots, {\bf X}_{N})& \mapsto & \left({\bf X}_2, \ldots, {\bf
X}_{N},
\frac{\Delta_{1}}{\Delta_{N}}, \ldots,
\frac{\Delta_{N-1}}{\Delta_{N}}\right)
\end{matrix}\]
It is clear that this is a birational map between  $\overline{k}^{n(N-1)}$ and
$\mathcal{V}$. Let $\mathcal{I}$ be the ideal of $\mathcal{V}$ in $k[{\bf X}_2,
\ldots, {\bf X}_N, a_1, \ldots, a_{N-1}]$. $\mathcal{I}$ is a prime ideal that
contains the polynomials $\Delta_i-a_i\Delta_{N}, i< N$ and -by Cramer's rule-
$f({\bf X}_{\ell})={\bf X}_\ell^{\gamma_N}+\sum_{i=1}^{N-1}a_i {\bf
X}_\ell^{\gamma_i}$, $2\leq \ell \leq N.$ Let ${\bf a}=\{a_1, \ldots,
a_{N-1}\}$. By construction, the field of rational functions of $\mathcal{V}$ is
 isomorphic to the field of fractions of the integer domain
\[\mathbb{L}=\textrm{Frac}\left(\frac{k[{\bf X}_2, \ldots, {\bf X}_N, {\bf
a}]}{\mathcal{I}}
\right)\simeq k({\bf X}_2, \ldots, {\bf X}_N).\]
In particular, $({\bf X}_2, \ldots, {\bf X}_N)$ is a transcendence basis of $k
\subset \mathbb{L}$ and the dimension of $\mathbb{L}$ is $n(N-1)$. For each
index $2\leq \ell\leq N$, choose one variable $X_{\ell, j_\ell}$ appearing in
$f({\bf X}_l)$. Denote by ${\bf X}_0=\{{\bf X}_2, \ldots, {\bf X}_N\}\setminus
\{X_{2, j_2}, \ldots, X_{N, j_N}\}$ the remaining variables $X_{i, j}$ not
chosen. As an element in $\mathbb{L}$, $X_{\ell, j_\ell}$ is algebraic over
$k({\bf X}_0, {\bf a})$, because $f(X_{\ell, j_\ell})\in \mathcal{I}$. So
$\mathbb{L}$ itself is an algebraic extension of $k({\bf X}_0, {\bf a})$. The
set $\{{\bf X}_0, {\bf a}\}$ is of cardinal $(n-1)(N-1)+(N-1)=n(N-1)$. So it is
a transcendence basis of $\mathbb{L}$ over $k$. In particular, this means that
the set $\{a_1, \ldots a_{N-1}\}$ is algebraically independent over
$\mathbb{L}$, and hence $\{\frac{\Delta_1}{\Delta_N},\ldots,
\frac{\Delta_{N-1}}{\Delta_N}\}$ is algebraically independent over $k$.
\end{proof}


\section{Proof of Theorem \ref{main}}\label{sec:proof}
With all the preliminaries given in Section \ref{bertini}, we can
prove the main theorem. It is clear that, if any of the three
conditions in the statement of Theorem \ref{main} fail to hold,
then $\VV$ factorizes.

Suppose then that these conditions are satisfied. By developing
$\VV$ as a polynomial in the variables indexed by ${\bf X}_1$, we
have the following
$$\VV=\sum_{i=1}^N(-1)^{\sigma_i}\Delta_i{\bf X}_1^{\gamma_i}, $$
with $\sigma_i\in\{0, 1\}$. Hence, we can regard $\VV$ as the
polynomial $P({\bf X}_1, \Gamma)$ specialized under $A_i\mapsto
(-1)^{\sigma_i}\Delta_i$.

As the family $\left((-1)^{\sigma_i}\Delta_i/\Delta_N\right)_{1\leq
i\leq N-1}$ is algebraically independent (due to Theorem
\ref{minors}), the polynomial ${\bf
X}_1^{\gamma_N}+\sum_{i=1}^{N-1}(-1)^{\sigma_i}\Delta_i/\Delta_N
{\bf X}_1^{\gamma_i}$ is generic among the polynomials of support
$\Gamma$, monic in $\gamma_1$. This means that, for almost every
$t_{ij}, 2\leq i\leq N, 1\leq j\leq N$,  the set of
zeroes of $\VV$ in $\overline{k}^n$ after setting
$X_{ij}\mapsto t_{ij}$ is
irreducible (by Proposition \ref{jouap}).

As a consequence of this, we get that the set of zeroes of $\VV$ is
irreducible in $\overline{k({\bf X}_2, \ldots, {\bf X}_N)}$, and hence -as in
Proposition \ref{desc}- $\VV$ must be the power
of an irreducible polynomial. By using again Proposition \ref{desc} and our
hypothesis, we conclude that $\VV$ is irreducible in $\overline{k({\bf X}_2,
\ldots, {\bf X}_N)}[X_1]$.

In order to show irreducibility in $\overline{k}[{\bf X}_1, {\bf
X}_2, \ldots, {\bf X}_N]$, we argue as follows: if it does factorize in
this ring, then it must have an irreducible factor depending only on
${\bf X}_2, \ldots, {\bf X}_N$. It cannot be a
monomial by the second hypothesis. So, it is a proper factor of positive degree
in -we can assume w.l.o.g.- ${\bf X}_2$ and degree zero in ${\bf X}_1$.  We then
have
$$
\VV=p({\bf X}_2,\ldots,{\bf X}_N)q({\bf X}_1,{\bf X}_2,\ldots,{\bf X}_N),$$
with $\deg_{{\bf X}_2}(p)>0.$
By
making the change of
coordinates ${\bf X}_2\leftrightarrow {\bf X}_1$,
we get
\begin{equation}\label{bar}
-\VV=p({\bf X}_1,{\bf X}_3,\ldots,{\bf X}_N)q({\bf X}_2,{\bf X}_1,\ldots,{\bf
X}_N).
\end{equation}
If $\deg_{{\bf X}_2}\left(q({\bf X}_1,{\bf X}_2,\ldots,{\bf X}_N)\right)>0$,
then minus the right hand side of (\ref{bar}) is
a factorization of $\VV$ with two factors os positive
degree in ${\bf X}_1$,  a contradiction with the
irreducibility over $\overline{k({\bf X}_2, \ldots, {\bf X}_N)}[{\bf
X}_1]$. So, we must actually have
$$
\VV=p({\bf X}_2,\ldots,{\bf X}_N)q({\bf X}_1,{\bf X}_3,\ldots,{\bf X}_N).
$$
But now, if we set ${\bf X}_1={\bf X}_2$ in $\VV$, we get
$$0=p({\bf X}_1,\ldots,{\bf X}_N)q({\bf X}_2,{\bf X}_3,\ldots,{\bf X}_N),$$
a contradiction with the fact that neither $p$ nor $q$ are zero. Hence, the
irreducibility of $\VV$ follows.

\section{The tropical approach}\label{sec:trop}

As in the introduction, the expression
$\VV=\sum_{i=1}^N(-1)^{\sigma_i}\Delta_i{\bf X}_1^{\gamma_i}$
corresponds to the development of the generalized Vandermonde
determinant with respect to the first row of its defining matrix. We
dehomogenize again this polynomial as $${\bf X}_1^{\gamma_N}
+\sum_{i=1}^{N-1} A_i {\bf X}_1^{\gamma_i},$$ where the $A_i$'s are
algebraically independent over $k$ by Theorem~\ref{minors}.

Given any rank one valuation: $v:\overline{k(A_1,\ldots,
A_{N-1})}\rightarrow \mathbb{R}$, we define Trop$(\VV)$, the tropicalization of
$\VV$, as
the closure of $\{\VV=0\}\subset \overline{k(A_1,\ldots,
A_{N-1})}^n$ under this valuation.

Let $\Lambda\subset\mathbb{R}^n$ be the convex hull of $\Gamma$. The values
$v(A_i)$
define a regular subdivision $Subdiv(\Lambda)$ that is
combinatorially dual to  Trop$(\VV)$
(\cite[Proposition 3.11]{Mik05}). In particular, by the duality, the
vertices of $Subdiv(\Lambda)$ correspond to the connected components
of $\mathbb{R}^n\setminus \textrm{Trop}(\VV)$, the edges of
$Subdiv(\Lambda)$ correspond to the facets of $\textrm{Trop}(\VV)$
and the two-dimensional polytopes of $Subdiv(\Lambda)$ correspond to
the ridges of $\textrm{Trop}(\VV)$. There are more cells, but we will focus only
on
these.

Every facet $F$ of $\textrm{Trop}(\VV)$ has
associated a \emph{multiplicity} as follows: let $e$ be the corresponding dual
edge of $F$ in $Subdiv(\Lambda)$. The multiplicity of $F$ is defined
as $\#(e\cap \mathbb{Z}^n)-1$, the integer length of $e$. With this definition,
the \emph{balancing condition} on the ridges of Trop$(\VV)$ holds:
given any such ridge $R$, let $F_1,\ldots, F_r$ be the facets containing
$R$ in their boundary,  $m_i$ be the multiplicity of $F_i$ and
$v_i$ the primitive integer normal vector to the affine hyperplane generated
by $F_i$ chosen with a compatible orientation. Then:
\[\sum_{i=1}^r m_iv_i=0\]
We refer to \cite{Mik05} or \cite{Tevelev-Sturmfels-Elimination} for more
background on this subject. We
will use the balancing condition to show the irreducibility of
Trop$(\VV)$.

\subsection*{Proof of Theorem \ref{irred_tropical}}
\begin{proof}
If one of the hypotheses of \ref{irred_tropical} is not fulfilled, then, it is
easy to find a field $k$ where $\VV$ factors, and hence Trop$(\VV)$ will be
reducible. Suppose then that the three conditions hold and let $k$ be any
field.

Consider ${\bf X}_1^{\gamma_N} +\sum_{i=1}^{N-1} A_i {\bf X}_1^{\gamma_i}$,
where $A_i$ are rational functions in $\{{\bf X}_2,\ldots, {\bf X}_N\}$,
algebraically independent over $k$. Hence, any function
$v:\{A_1,\ldots,A_{N-1}\}\rightarrow \mathbb{R}$, can be extended to a valuation
\[{\bf v}:\overline{k({\bf X}_2,\ldots,{\bf X}_N)}\rightarrow \mathbb{R}.\]

For our proof, we need a function that induces a regular
triangulation of $\Gamma$. We may take, for instance, any
appropriate generic infinitesimal perturbation of the standard
paraboloid lifting function:
\[v(A_i)=\sum_{j=1}^n((\gamma_{ij}+\epsilon_{ij})^2-\gamma_{Nj}^2).\]
This function induces a Delaunay triangulation of the set of exponents
$\{\gamma_1, \ldots, \gamma_N\}$ (See \cite{Handbook}).

The tropicalization of $\VV$ under this valuation is combinatorially dual to
this triangulation. So in particular, the ridges of Trop$(\VV)$ are always the
intersection of three facets, because their dual cell is always a triangle.
Moreover, for any such intersection, the compatible primitive vectors involved
in the balancing condition are two by two linearly independent.

Suppose that Trop$(\VV)=H_1 \cup H_2$. Let $F_1$ be a facet of
Trop($\VV$) and suppose that $F_1\subseteq H_1$. We want to prove
that Trop$(\VV)\subseteq H_1$ as sets of points. Let $R$ be any
ridge incident to $F_1$ and $F_2$, $F_3$ the other two facets
incident to $R$. Let $m_i$ be the weight of $F_i$ as a facet of
$H_1$, so $m_i=0$ if and only if $F_i$ is not a facet of $H_1$. Let
$v_i$ be the compatible primitive vector associated to $F_i$ and
$R$. Since $F_1\in H_1$, its weight must be a positive integer,
$m_1>0$.

From the balancing condition, we have that $m_1v_1+ m_2v_2
+m_3v_3=0$. Since $m_1>0$ and $v_1, v_2, v_3$ are pairwise linearly
independent vectors, it must happen that $m_2>0, m_3>0$. That is,
$F_2, F_3$ have positive weight, so they belong to $H_1$ as sets of
points. To sum up, for any facet $F$ of Trop($\VV$) belonging to
$H_1$ it happens that the facets that are ridge-connected to $F$
also belong to $H_1$. Now, since $\Gamma$ is not contained in a
line, it is known that Trop($\VV$) is ridge-connected, that is,
every two facets can be connected by a path of facets such that any
two of them that are consecutive have a common ridge.

We can then conclude by induction by showing that Trop$(\VV)=H_1$ as subsets of
$\mathbb{R}^n$. In particular, Trop$(\VV)$ cannot factorize as the union of two
different tropical hypersurfaces, set-theoretically.

However, it could still happen that Trop$(\VV)=H_1 \cup H_2$  with $H_1 =H_2$ as
sets, but being different only by the multiplicities of the facets. In that
case, let $m_i^1$,  $m_i^2$ be the multiplicity of $F_i$ as a facet of $H_1$ and
$H_2$ respectively. Then, it is easy to check that $m_i^1/m_i^2=p/q$ is a
rational constant that does not depend on the
facet. Thus, there are integer positive numbers $m_i^0$ such that
$m_i^1=k_1 m_i^0$, $m_i^2=k_2 m_i^0$, where $k_1,k_2\in
\mathbb{Z}_{>0}$ are constants not depending on the facet $i$. Hence,
the multiplicity of $F_i$ as a facet of $H$ is $(k_1+k_2)m_i^0$, and this imply
that every facet has a multiplicity which is a multiple of
$k_1+k_2\geq 2$. By duality, every edge of $Subdiv(\Lambda)$ will have an
integer length multiple of $(k_1+k_2)$. It follows that $d_\Gamma$
is a multiple of $k_1+k_2$, which contradicts the hypotheses.
\end{proof}


\end{document}